\documentclass[11pt,a4paper]{article}
\usepackage[utf8]{inputenc}
\usepackage[english]{babel}
\usepackage{amsmath}
\usepackage{amsfonts}
\usepackage{amssymb}
\usepackage{amsthm}
\usepackage{graphicx}
\usepackage{epsfig}
\usepackage{cite}
\usepackage[pagewise]{lineno}
\usepackage{dsfont}
\usepackage[all]{xy}
\usepackage{indentfirst}
\usepackage{url}
\usepackage{graphicx,color}
\usepackage{ragged2e}

\newtheorem{theorem}{Theorem}[section]
\newtheorem{lemma}[theorem]{Lemma}
\newtheorem{proposition}[theorem]{Proposition}
\newtheorem{Corolario}[theorem]{Corollary}
\newtheorem{teo}{Theorem}
\newtheorem{Theorem LO}{Theorem (Lichnerowicz-Obata)}
\newtheorem{Theorem A}{Theorem (Ma-Du)}

\newtheorem{definition}[theorem]{Definition}

\numberwithin{equation}{section}

\newcommand{\be}{\begin{eqnarray}}
\newcommand{\en}{\end{eqnarray}}
\newcommand{\no}{\nonumber}

\newcommand{\Addresses}{{
\bigskip
\footnotesize
\noindent A. C. Bezerra -- Goiano Federal Institute, Brazil \\
\textit{E-mail address:} \texttt{adriano.bezerra@ifgoiano.edu.br}
\medskip\\
T. Castro Silva -- University of Brasilia, Brazil \\
\textit{E-mail address:} \texttt{tarcisio@mat.unb.br}
\medskip\\
F. Manfio -- University of Sao Paulo, Brazil \\
\textit{E-mail address:} \texttt{manfio@icmc.usp.br}
}}

\title{Eigenvalue estimates and applications \\on weighted manifolds}

\date{}

\author{A. C. Bezerra\footnote{Corresponding author.},
T. Castro Silva\footnote{Partially supported by University of Brasilia Edital DPI/DPG nº 04/2024 and by FEMAT Edital 01/2024.}, F. Manfio\footnote{Supported by FAPESP, Grant 2022/16097-2.}}

\begin{document}

\maketitle



\begin{abstract}
We will present an estimate for the first eigenvalue of the Dirichlet and Neumann problems in terms of the Bakry-Émery Ricci curvature for a compact weighted manifold. As an application we will establish a stability condition for a $h$-minimal hypersurface.
\end{abstract}


\section{Introduction}


Manifolds with density have long appeared in mathematics, with more recent attention to differential geometry, where in many situations it is natural to consider a {\em weighted measure} of the form $e^{-h}dv$ on a Riemannian manifold $(M^{n}, g)$. Here, $h$ is a smooth function on $M$, referred to as the {\em weight function}. A {\em weighted manifold} $(M^{n}, g, e^{-h}dv)$, also knwon as a {\em manifold with density}, is defined as an $n$-dimensional Riemannian manifold $(M^{n}, g)$ equipped with a weighted volume form $e^{-h}dv$, where $dv$ is the volume element induced by the metric $g$.

The study of weighted manifolds was initiated by Bakry-Émery \cite{BE}, providing an extension of Riemannian geometry where many classical questions are being analyzed in recent years (see \cite{M,S,SS,VC,MW, YS, WW}). Grigori Perelman \cite{PG} introduced a functional that involves integrating the scalar curvature with respect to a weighted measure. The Ricci flow is, consequently, a gradient flow of such a functional. Motivated by this concept and supported  by Perelman's work  and the theory of optimal transport \cite{VC}, Ma and Du \cite{MD} extended the Reilly formula for the drifting Laplacian operator, associated with the weighted measure and Bakry-Émery Ricci tensor, on a compact Riemannian manifold with boundary. As an application, they derived  estimates for the first eigenvalue of the drifting Laplacian on manifolds with boundary.

A powerful tool in the study of weighted manifolds is the $\infty$-Bakry-Émery Ricci curvature tensor (Bakry-Émery Ricci curvature, for simplicity), which is defined as
\be\label{1.1}
Ric_{h} := Ric + \nabla^{2} h,
\en
where  $\nabla^{2} h$ is the Hessian of $h$ on $M$.  If $h$ is constant, $Ric_{h}$ reduces to the Ricci curvature, making the     Bakry-Émery Ricci curvature a generalization of the Ricci curvature. A Riemannian manifold $(M,g)$ is called a {\em gradient Ricci soliton} if $Ric_{h} = c g$ for some constant $c$. Thus, gradient Ricci solitons can
be viewed as weighted manifolds, representing a generalization of Einstein manifolds.

Over the last years, there has been an active research in the study of smooth manifolds with Bakry–Émery curvature bounded below. Much research has focused on establishing results analogous to those in the case of Ricci curvature bounded below, particularly those that link geometric conditions on the Ricci curvature with eigenvalue estimates for certain operators. One notable example is the following well-known result.


\

\noindent \textbf{Theorem (Lichnerowicz-Obata).} \textit{If the Ricci curvature of a compact Riemannian manifold $(M^{n},g)$ satisfies $Ric \geq a > 0$, then the first eigenvalue $\lambda_{1}$ of the Laplacian operator satisfies $\lambda_{1}\geq\dfrac{na}{n-1}$, with the equality holding if and only if $M$ is isometric to the unit $n$-dimensional sphere.}

\

The Lichnerowicz-Obata's theorem was extended by Reilly \cite{R} for the Dirichlet problem on a compact Riemannian manifold with smooth boundary, under the condition that the mean curvature of the boundary is nonnegative. More recently, Ma and Du \cite{MD} extended this result for a compact weighted  manifold with smooth boundary, addressing both the Dirichlet and Neumann problems (see equation \eqref{1.3}),  provided that the weighted mean curvature of the boundary is nonnegative  or if the boundary is convex, respectively.

Concerning to the weighted measure, the corresponding weighted Dirichlet energy functional is given by
\[
E_{h}(f)=\int_{M}|\nabla f|^{2}e^{-h}dv.
\]
Just as the Laplacian operator $\triangle$ is associated with the Dirichlet energy, the Euler-Lagrange operator of $E_{h}(f)$ is  known as the {\em weighted Laplacian} (also called the  {\em drifting Laplacian}) $\triangle_{h}$, which is given by
\be\label{1.2}
\triangle_{h}f:= \triangle f-\langle\nabla h,\nabla f\rangle.
\en
Note that it is a  second-order self-adjoint operator on $L^{2}(e^{-h}dv)$, the space of square integrable functions on $M$ with respect to the measure $e^{-h}dv$. Its fundamental importance arises from its relationship between fundamental gaps in the classical Laplacian operator on manifolds. We define the eigenvalue problems $(D)$ and $(N)$, which correspond to the Dirichlet and Neumann boundary conditions, respectively,
\be\label{1.3}
(D) \ \left\{ \begin{array}{ll}
\triangle_{h} f \ =- \lambda f \ \ \ in \ M,   & \\
\ \ \ \ \ f=0  \ \ \ \ \ on \ \partial M.  & \\
\end{array} \ \ \ \ \
(N) \ \left\{ \begin{array}{ll}
\triangle_{h} f \ =- \lambda f \ \ \ in \ M,   & \\
\ \ \  \frac{\partial f}{\partial \eta}=0  \ \ \ \ \ \ on \ \partial M. & \\
\end{array}   \right. \right.
\en
It is a well-known fact that the spectrum of problems $(D)$ and $(N)$ behaves like a non-decreasing sequence of real numbers in the following way:
\be
\no 0 < \lambda_{1} \leq \lambda_{2} \leq \cdot  \cdot  \cdot \rightarrow +\infty.
\en

Throughout this paper, we will denote an $n$-dimensional weighted manifold $(M^{n}, g, e^{-h}dv)$ by $M_{h}$, the weighted mean curvature of the boundary (which will be defined later) by $H^{\partial M_{h}}_{h}$, and the first eigenvalues of $(D)$ and $(N)$ by $\lambda_{1,D}$ and $\lambda_{1,N}$, respectively. Ma-Du's theorem can be stated as follows.

\

\noindent\textbf{Theorem (Ma-Du).} \textit{Let $M_{h}$ be an $n$-dimensional weighted compact Riemannian manifold with smooth boundary $\partial M_{h}$, and suppose that
\begin{equation}\label{1.4}
Ric_{h}\geq\left(\dfrac{|\nabla h|^{2}}{m-n}+a\right)g,
\end{equation}
for some constants $a>0$ and $m>n$. The following statements hold:}
\begin{enumerate}
\item[$(1)$] \textit{If the weighted mean curvature of the boundary satisfies  $H^{\partial M_{h}}_{h}\geq0$, then
$\lambda_{1,D}\geq\dfrac{ma}{m-n}$.}

\item[$(2)$] \textit{If $\partial M_{h}$ is convex, that is, the second fundamental form is non-negative, then
$\lambda_{1,N}\geq \dfrac{ma}{m-n}$.}
\end{enumerate}

Li and Wei \cite{LW} proved that this result is sharp, in the sense that the equality for $\lambda_{1,D}$ and $\lambda_{1,N}$ in Ma-Du's Theorem is achieved if and only if $M_{h}$ is isometric to a Euclidean hemisphere. Therefore, the results of Ma-Du and Li-Wei extend the rigidity theorems of Reilly \cite{R} and Escobar \cite{E}. Other authors such as Cheeger \cite{CH, CH2}, Aubin \cite{A}, Cheng \cite{CY} and Yau \cite{Y} obtained estimates for the first eigenvalue of the Laplacian related with geometric quantities such as volume, radius of injectivity, diameter and scalar curvature.

In this paper, we will obtain an important relationship between $\lambda_{1,D}$ and $\lambda_{1,N}$ and the modified Ricci tensor, which we will now define. We will also provide an interesting application for the estimate obtained. Our objective is to investigate how lower bounds for a suitably modified Ricci tensor influence the spectral behavior of geometric operators.

For a positive constant $c$ and $h\in C^{\infty}(M)$, we define the modified Bakry-Émery Ricci tensor  as
\be\label{Rhc}
 Ric_{h,c}=Ric_{h}-c|\nabla h|^{2}g,
\en
and the infimum of the  $Ric_{h,c}$ we will denote by
\be\label{infRhc}
k_{h,c}:=\displaystyle\inf_{p\in M}\lambda_{\min}(Ric_{h}-c|\nabla h|^{2}g)(p).
\en

 The tensor $Ric_{h,c}$ can be viewed as a modification of the Bakry-Émery Ricci tensor, incorporating the influence of the weight function $h$ through both its Hessian and gradient terms, and the quantity  $k_{h,c}$ provides a natural lower curvature bound associated with the weighted geometry of $M_{h}$. In particular, when $h$ is constant, this tensor reduces to the classical Ricci tensor, showing that $k_{h,c}$ extends the notion of a Ricci curvature lower bound to the weighted setting. Consequently,  $Ric_{h,c}$ and $k_{h,c}$ capture the combined effects of the Riemannian geometry and the underlying density structure. Our first result provides a lower bound for the first eigenvalues $\lambda_{1,D}$ and $\lambda_{1,N}$ in terms of the  tensor $Ric_{h,c}$.
\begin{teo} \label{teo:1}
Let $M_{h}$ be an  $n$-dimensional weighted compact Riemannian manifold with nonempty smooth boundary $\partial M_{h}$. Assume  there exists a constant $c$ satisfying  $Ric_{h,c}>0$. The following statements hold:
\begin{enumerate}
\item[$(1)$] If the weighted mean curvature of the boundary satisfies $H^{\partial M}_{h}\geq0$, then
$\lambda_{1,D}> k_{h,c}$.

\item[$(2)$] If $\partial M_{h}$ is convex, that is, the second fundamental form is non-negative, then
$\lambda_{1,N}> k_{h,c}$.

\end{enumerate}
\end{teo}

Since the estimates in $\lambda_{1,D}$ and $\lambda_{1,N}$ are strictly larger, an interesting question areises: what would be the optimal inequality involving the tensor $Ric_{h,c}$ and the weight function for any positive constant $c$ ? The Theorem 1 can be viewed as an extension of the eigenvalue estimates of Ma and Du. Indeed, their curvature  in \eqref{1.4} implies $Ric_{h,\frac{1}{m-n}}\geq ag>0$, hence our theorem applies. In contrast, our result requires only the positivity of the modified Bakry–Émery tensor $Ric_{h,c}$ for an arbitrary constant $c>0$, without introducing the dimensional parameter $n$. Our estimate is formulated directly in terms of the intrinsic quantity $Ric_{h,c}$ rather than an a priori lower bound $a$.

If $h$  is constant,  in \eqref{1.2} we have $\triangle_{h}=\triangle$ and the constant $k=\displaystyle\inf_{p\in M}\lambda_{\min}(Ric)(p)$. Thus the eingenvalue problems in \eqref{1.3} can be considered for the  Laplacian operator $\triangle$ in the Dirichlet and Neumann problems, leading to the following corollary.
\begin{Corolario}
Let $M$ be an $n$-dimensional compact Riemannian manifold with smooth boundary $\partial M$ and $Ric>0$. The following statements hold:
\item[$(1)$] If the  mean curvature of the boundary satisfies $H\geq0$, then
$\lambda_{1,D}> k$.

\item[$(2)$] If $\partial M$ is convex, that is, the second fundamental form is non-negative, then
$\lambda_{1,N}>k$.
\end{Corolario}

An important concept associated with weighted manifolds is that of stability (see \cite{BM},  \cite{CMZ2}, \cite{CMZ}, \cite{L}, among others), which will be described in Section \ref{sec:tability}. As an application of Theorem \ref{teo:1}, we obtain  a stability result for a compact $h$-minimal hypersurface with smooth boundary and nonnegative weighted mean curvature of the boundary, under a control condition in the infimum of the $Ric_{h,c}$.
Let $A$ be the second fundamental of the immersion.

\begin{teo} \label{teo:2}
Let $M_{h}$ be an $n$-dimensional compact $h$-minimal hypersurface  with  mean curvature $H\neq0$, isometrically immersed in a weighted manifold  $\overline{M}_{h}$. Suppose that the Hessian of weight function $h$ is a parallel tensor on $\overline{M}_{h}$, that is, $\overline{\nabla}(\overline{\nabla}^{2} h)\equiv 0$. If the mean curvature $H_{h}^{\partial M_{h}}\geq0$ and
\be\label{1.6}
k_{h,c} \geq 2\left(|A|^{2}+\dfrac{1}{H^{2}}\left(|\overline{\nabla}^{2}h|^{2}+|\nabla H|^{2}\right)\right),
\en
then $M_{h}$ is $L_{h}$-stable.
\end{teo}
The  Theorem 2 provides geometric consequences that relate to and refine several classic results known in the literature. In condition \eqref{1.6} we see how an infimum control on the modified Ricci tensor and the mean curvature of the immersion relate to the weight function to ensure its stability in the weighted context.

\section{Preliminaries and Proof of the Main Results}

In this section we recall some basic definitions and results that
are used in order to prove the results. We denote by $\nabla$ and
$Ric$ the Levi-Civita connection and the Ricci curvature tensor of $(M^{n}, g)$, respectively. Moreover, we denote $dv_{h}=e^{-h}dv$.

\subsection{The Ricci Bakry-Émery Curvature and the First Engenvalue}
Given an $n$-dimensional Riemannian manifold $(M^{n},g)$ and assuming that $f\in C^{\infty}(M)$, the well-known Bochner formula provides the following expression for the Laplacian of $f$
\be\no
\begin{aligned}
\frac{1}{2}\triangle |\nabla f|^{2}=|\nabla^{2} f|^{2}+\langle\nabla f,\nabla \triangle f \rangle
+Ric(\nabla f,\nabla f),
\end{aligned}
\en
where $\triangle f$, $\nabla f$ e $\nabla^{2} f$ are the Laplacian, the gradient and the hessian of $f$ on $M$, respectively. For a weighted manifold, the Bochner's formula takes a following form (cf. \cite{MD}):

\be\label{2.1}
\begin{aligned}\label{bochner_formula}
\frac{1}{2}\triangle_{h} |\nabla f|^{2}=|\nabla^{2} f|^{2}+\langle\nabla f,\nabla \triangle_{h} f \rangle
+Ric_{h}(\nabla f,\nabla f).
\end{aligned}
\en

The following result, which follows from \eqref{bochner_formula}, will be useful in the proof of Theorem \ref{teo:1}.

\begin{proposition}\label{Proposition 2.1}
Let $M_{h}$ be an $n$-dimensinal compact  weighted manifold with smooth boundary, and let $\eta$ be the unit normal vector field on the boundary $\partial M_{h}$. Given $f\in C^{\infty}(M_{h})$, we have
\begin{eqnarray} \label{2.2}
\begin{aligned}
\dfrac{1}{m}\int_{M}(\triangle_{h}f)^{2}dv_{h}
 &+\int_{M}\langle\nabla f,\nabla \triangle_{h} f \rangle dv_{h}-\dfrac{1}{m-n}\int_{M}\mid
 \nabla f\mid ^{2}\mid
 \nabla h\mid^{2}dv_{h} \\
&+\int_{M}Ric_{h}(\nabla f,\nabla f) dv_{h}\leq
\dfrac{1}{2}\int_{\partial M}\langle\nabla|\nabla f|^{2},\eta\rangle da_{h},
\end{aligned}
\end{eqnarray}
where $m>n$ is a constant.
\end{proposition}
\begin{proof}
Substituting $\triangle_{h}$ given in \eqref{1.2} into \eqref{bochner_formula}, we obtain
\begin{eqnarray*}
\begin{aligned}
\frac{1}{2}\left(2 |\nabla f|\triangle |\nabla f|+2 |\nabla| \nabla f||^{2}-\langle\nabla|\nabla f|^{2},\nabla h \rangle\right)
& =
|\nabla^{2} f|^{2}+\langle\nabla f,\nabla \triangle_{h} f \rangle
\\
& +  Ric_{h}(\nabla f,\nabla f),
\end{aligned}
\end{eqnarray*}
\noindent
that is,
\be\label{2.3}
\begin{aligned}
|\nabla| \nabla f||^{2}  =  |\nabla^{2} f|^{2}+\langle\nabla f,\nabla \triangle_{h} f \rangle
+Ric_{h}(\nabla f,\nabla f) -|\nabla f|\triangle_{h}|\nabla f|.
\end{aligned}
\en
Integrating \eqref{2.3} with respect to the measure $dv_{h}$, we obtain
\be\label{2.4}
\begin{aligned}
\int_{M}|\nabla| \nabla f||^{2} dv_{h} & = \int_{M}|\nabla^{2} f|^{2}dv_{h}+\int_{M}\langle\nabla f,\nabla \triangle_{h} f \rangle dv_{h} \\
&+\int_{M}Ric_{h}(\nabla f,\nabla f) dv_{h} -\int_{M}|\nabla f|\triangle_{h}|\nabla f| dv_{h}.
\end{aligned}
\en
On the other hand, it follows from the divergence theorem applied to the vector field $X=|\nabla f|\nabla|\nabla f|$ that
\be\label{2.5}
\begin{aligned}
\int_{M}|\nabla| \nabla f||^{2} dv_{h}+\int_{M}|\nabla f|\triangle_{h}|\nabla f| dv_{h}&=&\dfrac{1}{2}\int_{\partial M}\langle\nabla|\nabla f|^{2},\eta\rangle da_{h},
\end{aligned}
\en
where $da_{h}=e^{-h}da$ and  $da$ is the volume element induced by the metric
$g$ on $\partial M_{h}$.
Therefore, by \eqref{2.4} and \eqref{2.5} we get
\be\label{2.6}
\begin{aligned}
\frac{1}{2}\int_{\partial M}\langle\nabla|\nabla f|^{2},\eta\rangle da_{h} & =  \int_{M}|\nabla^{2} f|^{2}dv_{h}
 +\int_{M}\langle\nabla f,\nabla \triangle_{h} f \rangle dv_{h} \\
&+\int_{M}Ric_{h}(\nabla f,\nabla f) dv_{h}.
\end{aligned}
\en
Let us estimate the term $|\nabla^{2}f|^{2}$. In fact, taking a constant $m>n$ and using \eqref{1.2}, we have
\be\label{2.7}
 |\nabla^{2} f|^{2}\geq \dfrac{1}{n}(\triangle f)^{2}=\dfrac{1}{n}(\triangle_{h}f+\langle \nabla f,\nabla h\rangle)^{2}\geq \dfrac{(\triangle_{h}f)^{2}}{m}-\dfrac{\langle \nabla f,\nabla h\rangle^{2}}{m-n}.
\en
Substituting \eqref{2.7} into \eqref{2.6}, and using the Cauchy-Schwarz's formula, we obtain
\be
\begin{aligned}
\frac{1}{m}\int_{M}(\triangle_{h}f)^{2}dv_{h} &-\dfrac{1}{m-n}\int_{M}\mid
 \nabla f\mid ^{2}\mid
 \nabla h\mid^{2}dv_{h}+\int_{M}\langle\nabla f,\nabla \triangle_{h} f \rangle dv_{h} \\
&\no
+\int_{M}Ric_{h}(\nabla f,\nabla f) dv_{h}\leq  \dfrac{1}{2}\int_{\partial M}\langle\nabla|\nabla f|^{2},\eta\rangle da_{h},
\end{aligned}
\en
and this concludes the proof.
\end{proof}

Given a hypersurface $i:M^n\to\overline{M}^{n+1}$ of a Riemannian manifold $M^n$ into another Riemannian manifold $\overline{M}^{n+1}$, the restriction of a smooth function $h\in C^{\infty}(\overline{M}^{n+1})$ on $M^n$, which will also be denoted by $h$, defines a weighted measure $e^{-h}dv$ on $M^n$, thus yielding an induced smooth metric measure space $(M^{n}, g, e^{-h}dv)$. To avoid confusion, we will use a bar to denote the geometric objects in the immersion environment, while the same objects without a bar will refer to the hypersurface.

Given a point $p\in M^n$, recall that the {\em second fundamental form} $A$ of $M^n$ at $p$, identified with the shape operator at $p\in M^n$, is the linear operator $A:T_pM\to T_pM$ given by
\[
AX = \overline\nabla_X\nu,
\]
for every $X\in T_pM$, where $\nu$ is a smooth unit normal
vector field along $i$, around $p$. In a local orthonormal
frame $\{e_1,\ldots,e_n\}$ of $M^n$, the components of $A$
are denoted by
\[
a_{ij} = \langle Ae_i,e_j\rangle = \langle\overline\nabla_{e_{i}}\nu,e_j\rangle,
\]
with $1\leq i,j\leq n$. The {\em mean curvature} $H$ of $M^n$
at $p$ is defined by
\[
H=Tr(A)=\sum_{i=1}^{n} a_{ii}.
\]


With the above notations, we have the following

\begin{definition}
{\em The {\em weighted mean curvature} $H_{h}$ of the hypersurface $M_{h}$ is defined by
\begin{equation}\label{2.8}
H_{h} = H-\langle\overline{\nabla} h,\nu\rangle,
\end{equation}
and  $M_{h}$ is called a $h$-{\em minimal hypersurface} if its mean curvature $H$ satisfies the condition
\[
H=\langle\overline{\nabla}h,\nu\rangle.
\]}
\end{definition}

It is known that $M_{h}$  is $h$-minimal if and only if it is a critical point of the
weighted volume functional  $V_{h}({M})$, defined by
\begin{equation*}
V_{h}(M):=\int_{M}e^{-h}dv.
\end{equation*}
Furthermore, $M_{h}$ being $h$-minimal in $(\overline{M}_{h},\overline{g})$ is equivalent to $(M,i^{\ast}\tilde{g})$ being minimal
in $(\overline{M},\tilde{g})$, where $\tilde{g}$ is the conformal metric given by $\tilde{g} = e^{-\frac{2h}{n}}\overline{g}$.

\vspace{.2cm}

From now on, we will assume that $M_{h}$ is a two-sided hypersurface, meaning there exists a smooth unit normal vector field $\nu$ along $M_{h}$.

\begin{lemma}\label{Lemma 2.3}
Let $i:M^{n}\to \overline{M}^{n+1}$ be a hypersurface. If
$f:\overline{M}\to\mathbb{R}$ is a smooth function, then
for each point $p\in M$, the following holds:
\be\label{2.9}
\begin{aligned}
 \overline{\triangle }f=\triangle f+Hf_{\nu}+\overline{\nabla}^{2}f(\nu,\nu).
\end{aligned}
\en
\end{lemma}

The  Lemma \ref{Lemma 2.3} can be found in \cite{CaMa}.
The version of \eqref{2.9}  can be easily adapted to $\overline{\triangle} _{h}f$ as follows:
\be\label{2.10}
\begin{aligned}
 \overline{\triangle }_{h}f=\triangle _{h}f+H_{h}f_{\nu}+\overline{\nabla}^{2}f(\nu,\nu).
\end{aligned}
\en

Since the boundary $\partial M _{h}$ can be considered as a
hypersurface isometrically immersed in $M _{h}$, we will use
a bar to represent the geometric entities of $M _{h}$ to prove the Theorem 1.

\noindent
\textbf{\textit{Proof of Theorem 1.}}
For the item $(1)$, it follows from \eqref{2.2} that
\begin{eqnarray*}
\begin{aligned}
\frac{1}{m}&\int_{M}(\overline{\triangle}_{h}f)^{2}dv_{h}
 +\int_{M}\langle\overline{\nabla} f,\overline{\nabla} \bar{\triangle}_{h} f \rangle dv_{h}-\dfrac{1}{m-n}\int_{M }\mid
 \overline{\nabla} f\mid ^{2}\mid
 \overline{\nabla} h\mid^{2}dv_{h}\\
 &+\int_{M}\overline{R}ic_{h}(\overline{\nabla} f,\nabla f) dv_{h}\leq
\dfrac{1}{2}\int_{\partial M }\langle\overline{\nabla}|\overline{\nabla} f|^{2},\eta\rangle da_{h}.
\end{aligned}
\end{eqnarray*}
We assume that $f$ is an eigenfunction associated with the first nonzero eigenvalue $\lambda_{1,D}$ of the  Dirichlet problem $(D)$ in \eqref{1.3}. Since $f\mid_{\partial M}=0$, we have
\be\label{2.11}
\no\dfrac{1}{2}\int_{\partial M }\langle\overline{\nabla}|\overline{\nabla} f|^{2},\eta\rangle da_{h} &=&\int_{\partial M }\overline{\nabla}^{2}f(\eta,\overline{\nabla} f) da_{h}\\
 &=&\int_{\partial M }\overline{\nabla}^{2}f(\eta, \nabla f+f_{\eta}\eta) da_{h}\\
\no &=&\int_{\partial M }f_{\eta}\overline{\nabla}^{2}f(\eta, \eta) da_{h}.
\en
Since $\overline{\triangle}_{h}f=-\lambda_{1,D} f$  and   $f\mid_{\partial M _{h}}=0$, we have  $\triangle_{h}f=0$. Regardless
of the assumption that $H^{\partial M}_{h}\geq0$, and taking into account \eqref{2.10} and \eqref{2.11}, we obtain
\be\label{2.12}
 \dfrac{1}{2}\int_{\partial M }\langle\overline{\nabla}|\overline{\nabla} f|^{2},\eta\rangle da_{h}
=\int_{\partial M }( f_{\eta}\overline{\triangle}_{h} f- f_{\eta}\triangle_{h} f-H^{\partial M}_{h} f^{2}_{\eta}) da_{h}\leq0.
\en
Consequently, using \eqref{2.12}, the inequality \eqref{2.2} becomes
\begin{eqnarray} \label{2.13}
\begin{aligned}
\frac{\lambda_{1,D}^{2}}{m}&\int_{M }f^{2}dv_{h}-\dfrac{1}{m-n}\int_{M }\mid
\overline{ \nabla} f\mid ^{2}\mid
 \overline{\nabla} h\mid^{2}dv_{h }\\
&-\lambda_{1,D}\int_{M }\mid\overline{\nabla} f\mid^{2} dv_{h}
+\int_{M }\overline{R}ic_{h}(\overline{\nabla }f,\overline{\nabla} f) dv_{h}\leq0.
\end{aligned}
\end{eqnarray}
Now, we claim that $\lambda_{1,D}> k_{h,c}$. In fact, suppose that this inequality does not hold.
By the definition of $\overline{R}ic_{h,c}$ we have
\be
\begin{aligned}
\nonumber \nonumber \overline{R}ic_{h,c}(\overline{\nabla} f,\overline{\nabla} f)= \overline{R}ic_{h}(\overline{\nabla} f,\overline{\nabla} f)
-c\mid\overline{\nabla}h\mid^{2}\mid\overline{\nabla}f\mid^{2}.
\end{aligned}
\en
 The inequality \eqref{2.13} leads to
\be\label{2.14}
\begin{aligned}
 \frac{\lambda_{1,D}^{2}}{m}\int_{M}f^{2}dv_{h}-\dfrac{1}{m-n}\int_{M}\mid
\overline{ \nabla} f\mid ^{2}\mid\overline{\nabla }h\mid^{2}dv_{h}
 - \lambda_{1,D}\int_{M }\mid\overline{\nabla} f\mid^{2} dv_{h} \\
 +c\int_{M}\mid
\overline{ \nabla} f\mid ^{2}\mid\overline{\nabla }h\mid^{2}dv_{h}+\int_{M }\overline{R}ic_{h,c}(\overline{\nabla }f,\overline{\nabla} f) dv_{h}\leq0.
\end{aligned}
\en
Since that $ k_{h,c}|\overline{\nabla} f|^{2}\leq\overline{R}ic_{h,c}(\overline{\nabla }f,\overline{\nabla} f)$, choosing $m\in \mathbb{R}$ such that $m\geq n+\dfrac{1}{c}$, it follows from \eqref{2.14} that
\be
\begin{aligned}
\nonumber\dfrac{\lambda_{1,D}^{2}}{m}\int_{M}f^{2}dv_{h}+(k_{h,c} - \lambda_{1,D})\int_{M }\mid\overline{\nabla} f\mid^{2} dv_{h}\leq0,
\end{aligned}
\en
from which we conclude that
\be
\begin{aligned}
\nonumber\dfrac{\lambda_{1,D}^{2}}{m}\int_{M}f^{2}dv_{h}\leq0,
\end{aligned}
\en
and this is a contradiction, since $\lambda_{1,D}$ is always positive. For item  (2), we assume that $f$ is an eigenfunction associated with the first nonzero
eigenvalue $\lambda_{1,N}$ of the Neumann problem $(N)$ in (1.3). From the hypothesis, we have $f_{\eta}=0$, and thus $\overline{\nabla} f=\nabla f$. Therefore, from \eqref{2.11}, we have
\be
\no \dfrac{1}{2}\int_{\partial M }\langle\overline{\nabla}|\overline{\nabla} f|^{2},\eta\rangle da_{h} &=&\dfrac{1}{2}\int_{\partial M }\langle\overline{\nabla}|\nabla f|^{2},\eta\rangle da_{h}\\
\no &=&\int_{\partial M }\overline{\nabla}^{2}f(\eta,\nabla f) da_{h}\\
\no &=&\int_{\partial M}\langle \overline{\nabla}_{\nabla f} \overline{\nabla }f, \eta\rangle  da_{h}\\
\no &=&\int_{\partial M}(\nabla f\langle  \nabla f, \eta\rangle- \langle \nabla f,\overline{\nabla}_{\nabla f} \eta\rangle )da_{h}.
\en
Note that $\langle  \nabla f, \eta\rangle=0$ and $\langle \nabla f,\overline{\nabla}_{\nabla f} \eta\rangle=A^{\partial M _{h}}(\nabla f,\nabla f)$, where $A^{\partial M _{h}}$ is the second fundamental form of $\partial M _{h}$. Therefore,

\be
\no \dfrac{1}{2}\int_{\partial M }\langle\bar{\nabla}|\bar{\nabla} f|^{2},\eta\rangle da_{h}=- \int_{\partial M _{h}}A^{\partial M}(\nabla f,\nabla f) da_{h}\leq0.
\en
The remainder of the proof preceeds as in the conclusion of item (1), and this concludes the proof. \qed

\subsection{A stability condition for $h$-minimal hypersurfaces}
\label{sec:tability}

Given a hypersurface $i:M^{n}\to\overline{M}^{n+1}$, the $L_{h}$-{\em stability operator} of $M _{h}$ is given by
\be\label{2.16}
L_{h}:= \triangle_{h} + |A|^{2} + \overline{R}ic_{h} (\nu, \nu),
\en
where $|A|^{2}$ denotes the square of norm of the second fundamental form of $M _{h}$ and $\nu$ is an unit normal vector field to $M _{h}$.

\begin{definition} \label{definition 3.1}
{\em A two-sided $h$-minimal hypersurface $M _{h}$ is said to be $L_{h}$-{\em stable} if, for any compactly supported smooth function $\varphi\in C^{\infty}_{0}(M _{h})$, the following inequality  holds
\begin{equation*}
-\int_{M } \varphi L_{h}\varphi e^{-h} dv\geq 0.
\end{equation*}
Equivalently,
\begin{equation*}
\int_{M } [|\varphi|^{2}-(|A|^{2}+\overline{R}ic_{h}(\nu,\nu))\varphi^{2}]e^{-h} dv\geq 0.
\end{equation*}}
\end{definition}

The stability of $M _{h}$ implies that the second variation of the weighted volume of an $h$-minimal hypersurface $M _{h}$ is nonnegative. The concept of $L_{h}$-stability has has been extensively developed in recent years (see \cite{CMZ2}, \cite{L} and references therein) and important results relating to the stability of $h$-minimal hypersurfaces with geometric and topological conditions have been obtained. Analogous to the concept of minimal immersion, it is known that an $h$-minimal hypersurface $(M _{h}, g)$ is $L_{h}$-stable if and only
if $(M, i^{\ast}\tilde{g})$ is stable as a minimal hypersurface in $(\overline{M},\tilde{g})$.

\vspace{.2cm}

It is convenient to use the following notation. If $\{e_1,\ldots,e_n\}$ is a local orthonormal frame on $M _{h}$ and $S=(S_{k_{1}\ldots k_{s}})$ is an $(s,0)$-tensor on $M _{h}$, the components of the covariant derivative $\nabla S$ of $S$ will be denoted by $S_{k_{1}\ldots k_{s},l}$, that is,
\[
S_{k_{1}\ldots k_{s},l}= ({\nabla}_{e_{l}}S)(e_{k_{1}},\ldots, e_{k_{s}}) = (\nabla S)(e_{l},e_{k_{1}},\ldots,e_{k_{s}}).
\]
Moreover, we say that $S$ is {\em parallel} on $M$ if $\nabla S  \equiv0$.


As an application of Theorem \ref{teo:1}, we will see how the stability of a compact $h$-minimal hypersurface $M _{h}$ can be related to its curvature. In order to do this, we will prove a result regarding the $L_{h}$-stability operator applied to the mean curvature of the hypersurface.

\begin{proposition}
{\em Let $(M _{h},g)$ be an $h$-minimal hypersurface into a smooth metric measure space $\overline{M}_{h}$. Then, the mean curvature $H$ of $M _{h}$ satisfies
\begin{eqnarray} \label{2.17}
\begin{aligned}
L_{h}(fH)&=f\left(2\sum_{i=1}^{n}(\overline{\nabla}^{3}h)_{i\nu i}-\sum_{i=1}^{n}(\overline{\nabla}^{3}h)_{\nu ii}+2\sum_{i,k=1}^{n}a_{ik}(\overline{\nabla}^{2}h)_{ki}\right)\\
&+ 2\sum_{i=1}^{n}e_{i}(H)\langle\overline{\nabla}f,e_{i}\rangle+H\triangle_{h}f,
\end{aligned}
\end{eqnarray}
for every $f\in C^{\infty}_{0}(M _{h})$, where $\{e_1,\ldots,e_n\}$ is a local orthonormal frame field on $M _{h}$ and $\nu$ is an unit normal vector field along $M _{h}$.}
\end{proposition}
\begin{proof}
Choose a local orthonormal frame $\{e_1,\dots,e_n,e_{n+1}\}$ for $\overline{M} _{h}$ such that $\{e_1,\dots,e_n\}$ are tangent to $M _{h}$ and $e_{n+1}=\nu$ is normal to $M _{h}$. For simplicity, we replace $\nu$ for the subscript $n+1$ in the components of the tensors on $\overline{M} _{h}$; for instance
$\overline{R}_{\nu ikj}= \overline{R}m(\nu, e_{i}, e_{k}, e_{j})$ and $(\overline{\nabla}^{2}h)_{\nu,i}=(\overline{\nabla}^{2}h)(\nu,e_{i})$. Since $M _{h}$ is  an $h$-minimal hypersurface, multiplying \eqref{2.8} by $f\in C^{\infty}_{0}(M _{h})$ and differentiating, we obtain
\be
\no e_{i}(fH)&=&e_{i}\left(f\langle\overline{\nabla} h,\nu\rangle\right)\\
\no &=&f(\langle\overline{\nabla}_{e_{i}}(\overline{\nabla} h),\nu\rangle+\langle\overline{\nabla} h,\overline{\nabla}_{e_{i}}\nu\rangle)+H\langle \overline{\nabla}f,e_{i}\rangle\\
\no &=&f(\overline{\nabla}^{2}h(e_{i},\nu)+\sum_{k=1}^{n}a_{ik}\langle\overline{\nabla} h,e_{k}\rangle)+Hf_{i},
\en
for $1 \leq i \leq n$. Then for $1\leq i,j\leq n$, one has
\begin{eqnarray} \label{2.18}
\begin{aligned}
e_{j}e_{i}(fH)& = e_{j}\left(f\left[\overline{\nabla}^{2}h(e_{i},\nu)+\sum_{k=1}^{n}a_{ik}\langle\overline{\nabla} h,e_{k}\rangle\right]+Hf_{i}\right) \\
&=f\left[e_{j}(\overline{\nabla}^{2}h(e_{i},\nu))+\sum_{k=1}^{n}e_{j}(a_{ik})h_{k}+\sum_{k=1}^{n}a_{ik}e_{j}(\langle\overline{\nabla} h,e_{k}\rangle)\right] \\
&+f_{j}\left[\overline{\nabla}^{2}h(e_{i},\nu)+\sum_{k=1}^{n}a_{ik}\langle\overline{\nabla} h,e_{k}\rangle\right]+e_{j}(H)f_{i}+He_{j}(\langle \overline{\nabla}f,e_{i}\rangle).
\end{aligned}
\end{eqnarray}
Now, fix a point $p\in M_{h}$ and choose the local orthonormal frame  $\{e_1,\ldots,e_{n}\}$ so that $\nabla_{e_{i}}e_{j}(p)=\overline{\nabla}^{\top}_{e_{i}}e_{j}(p) = 0$, $1\leq i,j\leq n$.  Then, at the point $p$, we obtain from the Codazzi equation the following relations:
\be\label{2.19}
\raggedleft
\no( \nabla^{2} (fH))(e_{j},e_{i})&=&\langle\nabla_{e_{j}}\nabla  (fH), e_{i}\rangle\\
&=&e_{j}\langle\nabla  (fH), e_{i}\rangle-\langle\nabla (fH), \nabla_{e_{j}}e_{i}\rangle\\
\no&=&e_{j} e_{i} (fH),
\en
\begin{equation} \label{2.20}e_{j}
(\overline{\nabla}^{2}h(e_{i},\nu))=(\overline{\nabla}^{3}h)_{j\nu i}+\sum_{k=1}^{n}a_{jk}(\overline{\nabla}^{2}h)_{ki} -
a_{ji}(\overline{\nabla}^{2}h)_{\nu\nu},
\end{equation}
\begin{equation} \label{2.21}
\sum_{k=1}^{n}e_{j}(a_{ik})h_{k}=\sum_{k=1}^{n}a_{ij,k}h_{k}+\sum_{k=1}^{n}\overline{R}_{\nu ikj}h_{k},
\end{equation}
and
\begin{equation} \label{2.22}
\sum_{k=1}^{n}a_{ik}e_{j}(\langle\overline{\nabla} h,e_{k}\rangle)=\sum_{k=1}^{n}a_{ik}(\overline{\nabla}^{2}h)_{jk}-\sum_{k=1}^{n}a_{ik}a_{jk}h_{\nu}.
\end{equation}
On the other hand, the following holds on $M_{h}$:
\be
\no (\overline{\nabla}^{3}h)_{i\nu j}&=&(\overline{\nabla}^{2}h)_{\nu j,i}=(\overline{\nabla}^{2}h)_{j\nu ,i}\\
\no & = &(\overline{\nabla}^{2}h)_{ji,\nu}+\sum_{k=1}^{n+1}h_{k}\overline{R}_{kj\nu i}\\
\no & = &(\overline{\nabla}^{3}h)_{\nu ji}+h_{\nu}\overline{R}_{\nu i\nu j}+\sum_{k=1}^{n}h_{k}\overline{R}_{\nu ikj}.
\en
Thus, we have
\be\label{2.23}
\sum_{k=1}^{n}h_{k}\overline{R}_{\nu ikj}=
 (\overline{\nabla}^{3}h)_{i\nu j}-(\overline{\nabla}^{3}h)_{\nu ji}-h_{\nu}\overline{R}_{\nu i\nu j}.
 \en
Substituting \eqref{2.19}-\eqref{2.23} into \eqref{2.18} and taking into account that $h_{\nu}= H$, we have at $p$ and for $1 \leq i, j \leq n$,
\be
\no( \nabla^{2} (fH))(e_{j},e_{i}) &=&f\left[(\overline{\nabla}^{3}h)_{j\nu i}+ (\overline{\nabla}^{3}h)_{i\nu j}-(\overline{\nabla}^{3}h)_{\nu ji}+\sum_{k=1}^{n}a_{jk}(\overline{\nabla}^{2}h)_{ki}\right.\\
\no & &\left.
+\sum_{k=1}^{n}a_{ik}(\overline{\nabla}^{2}h)_{jk}+\sum_{k=1}^{n}a_{ij,k}h_{k}-a_{ji}(\overline{\nabla}^{2}h)_{\nu\nu}
-H\overline{R}_{i\nu j\nu }\right.\\
\no&& \left. -\sum_{k=1}^{n}a_{ik}a_{jk}H\right]+f_{j}\left[\overline{\nabla}^{2}h(e_{i},\nu)+\sum_{k=1}^{n}a_{ik}\langle\overline{\nabla} h,e_{k}\rangle\right]\\
\no&&+e_{j}(H)f_{i}+He_{j}(\langle \overline{\nabla}f,e_{i}\rangle).
\en
Taking the trace, we obtain
\be\label{2.24}
\no\triangle(fH)&=&f\left[2\sum_{i=1}^{n}(\overline{\nabla}^{3}h)_{i\nu i}-\sum_{i=1}^{n}(\overline{\nabla}^{3}h)_{\nu ii}+2\sum_{i,k=1}^{n}a_{ik}(\overline{\nabla}^{2}h)_{ki}\right.\\
 & &\left.
\langle\nabla h,\nabla H\rangle-\overline{R}ic_{h}(\nu,\nu)H-|A|^{2}H\right]\\
&& \no+\sum_{i=1}^{n}f_{i}\left[\overline{\nabla}^{2}h(e_{i},\nu)+\sum_{k=1}^{n}a_{ik}\langle\overline{\nabla} h,e_{k}\rangle\right]\\
\no&&+\sum_{i=1}^{n}e_{i}(H)f_{i}+\sum_{i=1}^{n}He_{i}(\langle \overline{\nabla}f,e_{i}\rangle).
\en
Since $p\in M_{h}$ is arbitrary and \eqref{2.24} is independent of the choice of the frame, by the expression of $L_{h}$ in \eqref{2.16} we have
\be\label{2.25}
\no L_{h}(fH)&=&f\left[2\sum_{i=1}^{n}(\overline{\nabla}^{3}h)_{i\nu i}-\sum_{i=1}^{n}(\overline{\nabla}^{3}h)_{\nu ii}+2\sum_{i,k=1}^{n}a_{ik}(\overline{\nabla}^{2}h)_{ki}\right]\\
&& +\sum_{i=1}^{n}f_{i}\left[\overline{\nabla}^{2}h(e_{i},\nu)+\sum_{k=1}^{n}a_{ik}\langle\overline{\nabla} h,e_{k}\rangle\right]\\
\no&&+\sum_{i=1}^{n}e_{i}(H)\langle\overline{\nabla}f,e_{i}\rangle+\sum_{i=1}^{n}He_{i}(\langle \overline{\nabla}f,e_{i}\rangle)-H\langle\nabla h,\nabla f\rangle.
\en
We will rearrange the last five terms of \eqref{2.25}. More precisely, for $1\leq i\leq n$, one has
\begin{eqnarray} \label{2.26}
\begin{aligned}
\nabla^{2}f(e_{i},e_{i}) &= \langle\nabla_{e_{i}}\nabla f,e_{i}\rangle= e_{i}(\langle\nabla f,e_{i}\rangle) -
\langle\nabla f,\nabla_{e_{i}}e_{i}\rangle \\
& = e_{i}(\langle\nabla f,e_{i}\rangle).
\end{aligned}
\end{eqnarray}
Since
$\overline{\nabla}f=\nabla f+\overline{\nabla}^{\perp}f$ lends to $\langle\overline{\nabla} f,e_{i}\rangle=\langle\nabla f,e_{i}\rangle$,
one has
\be\label{2.27}
&&\sum_{i=1}^{n}He_{i}(\langle \overline{\nabla}f,e_{i}\rangle)-H\langle\nabla h,\nabla f\rangle=H\triangle_{h} f.
\en
On the other hand,
\begin{eqnarray} \label{2.28}
\begin{aligned}
\sum_{i=1}^{n}f_{i}\overline{\nabla}^{2}h(e_{i},\nu)+\sum_{i=1}^{n}e_{i}(H)\langle\overline{\nabla}f,e_{i}\rangle &= 2\sum_{i=1}^{n}e_{i}(H)\langle\overline{\nabla}f,e_{i}\rangle \\
&-\sum_{i,k=1}^{n}f_{i}a_{ik}
\langle\overline{\nabla}h,e_{k}\rangle.
\end{aligned}
\end{eqnarray}
Finally, applying the equations \eqref{2.26}-\eqref{2.28} into \eqref{2.25}, we get \eqref{2.17}, and this concludes the proof.
\end{proof}

We are now in position to prove Theorem \ref{teo:2}.

\

\noindent\textbf{\textit{Proof of Theorem \ref{teo:2}}}. Multiplying \eqref{2.17} by $fH$, we obtain
\begin{eqnarray} \label{2.29}
\begin{aligned}
fHL_{h}(fH) &=f^{2}H\left[2\sum_{i=1}^{n}(\overline{\nabla}^{3}h)_{i\nu i}-\sum_{i=1}^{n}(\overline{\nabla}^{3}h)_{\nu ii}+2\sum_{i,k=1}^{n}a_{ik}(\overline{\nabla}^{2}h)_{ki}\right]\qquad \\
&+2fH\sum_{i=1}^{n}e_{i}(H)\langle\overline{\nabla}f,e_{i}\rangle+fH^{2}\triangle_{h}f.
\end{aligned}
\end{eqnarray}
By the assumption, one has $\overline{\nabla}(\overline{\nabla}^{2} h)\equiv 0$, and this implies that
\[
\sum_{i=1}^{n}(\overline{\nabla}^{3}h)_{i\nu i}
=\sum_{i=1}^{n}(\overline{\nabla}^{3}h)_{\nu ii}=0,
\]
which provides us with
\begin{eqnarray} \label{2.30}
\begin{aligned}
fHL_{h}(fH) &= 2f^{2}H\sum_{i,k=1}^{n}a_{ik}(\overline{\nabla}^{2}h)_{ki} + 2fH\sum_{i=1}^{n}e_{i}(H)\langle\overline{\nabla}f,e_{i}\rangle \\
&+fH^{2}\triangle_{h}f.
\end{aligned}
\end{eqnarray}
Setting
\be\no
\begin{aligned}
a:=2f^{2}H\sum_{i,k=1}^{n}a_{ik}(\overline{\nabla}^{2}h)_{ik}, \ \ b:=2fH\sum_{i=1}^{n}e_{i}(H)\langle\overline{\nabla}f,e_{i}\rangle \ \ c:=fH^{2}\triangle_{h}f.
\end{aligned}
\en
We obtain from Young's inequality the following estimate
\be\no
a&\leq &2\sum_{i,k=1}^{n}f^{2}|H||a_{ik}||(\overline{\nabla}^{2}h)_{ik}|\\
\no&\leq&2\sum_{i,k=1}^{n}\left( \dfrac{f^{2}H^{2}a_{ik}^{2}}{2}+ \dfrac{f^{2}(\overline{\nabla}^{2}h)^{2}_{ik}}{2}\right)\\
\no&=&| A |^{2}H^{2}f^{2}+| \overline{\nabla}^{2}h|^{2}f^{2}.
\en
The expression $b$ can be rewritten as
\[
\no b:=2fH\sum_{i=1}^{n}e_{i}(H)\langle\overline{\nabla}f,e_{i}\rangle=2fH\langle\nabla H,\nabla f\rangle,
\]
and using integration by parts in $c$, we have
\be\no
\int_{ M}c \ dv_{h}=\int_{ M}H^{2}f\triangle_{h} fdv_{h}=-\int_{ M}2Hf\langle \nabla H,\nabla f\rangle dv_{h}-\int_{ M}H^{2}|\nabla f |^{2}dv_{h}.
\en
Finally, using $a$, $b$ and $c$, we obtain the following integral inequality for \eqref{2.30}:
\begin{eqnarray} \label{2.31}
\begin{aligned}
-\int_{ M}fHL_{h}(fH)dv_{h}&\geq -\int_{ M}| A |^{2}H^{2}f^{2}dv_{h}- \int_{M}| \overline{\nabla}^{2}h|^{2}f^{2}dv_{h} \\
&+\int_{ M}H^{2}|\nabla f |^{2}dv_{h}.
\end{aligned}
\end{eqnarray}
Now, note that the first eigenvalue $\lambda_{1,D}$ of $D$ in \eqref{1.3} can be characterized variationally in the following way:
\be\label{2.32}
 \lambda_{1,D}\int_{M}f^{2}dV_{h}&\leq &\int_{ M}|\nabla f |^{2}dV_{h}, \ \ \forall \ f\in C^{\infty}_{0}(M_{h}).
\en
Replacing $f$ by $fH$ in \eqref{2.32}, it follows from Young's inequality that
\be\label{2.33}
 \lambda_{1,D}\int_{ M}f^{2}H^{2}dV_{h}&\leq& 2\left(\int_{ M}f^{2}|\nabla H|^{2} dV_{h}+ \int_{ M}H^{2}|\nabla f |^{2} dV_{h}\right).
\en
Joining \eqref{2.31} and \eqref{2.33}, we get
\be
\no -\int_{M}fHL_{h}(fH)dv_{h}&\geq & -\int_{ M}| A |^{2}H^{2}f^{2}dv_{h}- \int_{ M}| \overline{\nabla}^{2}h|^{2}f^{2}dv_{h}\\
\no&&+\dfrac{1}{2}\int_{ M}\lambda_{1,D}H^{2}f^{2}dv_{h}- \int_{ M}f^{2}|\nabla H|^{2} dv_{h}.
\en
Since $M_{h}$ is compact and $Ric_{h,c}>0$, as the mean curvature of the boundary satisfies $H_{h}^{\partial M_{h}}\geq0$ the Theorem \ref{teo:1} tells us that $\lambda_{1,D}>k_{h,c}$ for an appropriate positive constant $c$. Therefore, the inequality described above becomes
 \begin{eqnarray}
\begin{aligned}
-\int_{M}fHL_{h}(fH)dv_{h}&\geq \int_{M}\left(\dfrac{1}{2}k_{h,c}-| A|^{2}-\dfrac{1}{H^{2}}| \overline{\nabla}^{2}h|^{2}\right.\\
& \left.-\dfrac{1}{H^{2}}|\nabla H|^{2} \right)H^{2}f^{2}dv_{h}.
\end{aligned}
\end{eqnarray}
Then, from \eqref{1.6}, one has
\be
\no\dfrac{1}{2}k_{h}-| A|^{2}-\dfrac{1}{H^{2}}| \overline{\nabla}^{2}h|^{2}-\dfrac{1}{H^{2}}|\nabla H|^{2}\geq0,
\en
which allows us to conclude that
\be
\no -\int_{ M}fHL_{h}(fH)dv_{h}\geq  0,  \ \ \forall \  f\in C^{\infty}_{0}(M_{h}).
\en
Given $\varphi \in C^{\infty}_{0}(M_{h})$ and since that $H\neq 0$,  taking $f=\dfrac{1}{H}\varphi$, we have
\be
\no -\int_{ M}\varphi L_{h}(\varphi)dv_{h}\geq  0,  \ \ \forall \  \varphi\in C^{\infty}_{0}(M_{h}),
\en
and according to Definition \ref{definition 3.1}, we conclude that $M_{h}$ is $h$-stable. \qed

\

\

\noindent {\Large{Acknowledgments}}

\

 \noindent  ACB is grateful to the Goiano Federal Institute for the necessary conditions for the research internship, and to the Department of Mathematics of the University of Brasilia for its hospitality while this paper was being prepared. We express our
sincere gratitude to Keti Tenenblat for invaluable suggestions.

\Addresses

\end{document}